\documentclass[a4paper,10pt]{article}
\usepackage[english]{babel}
\usepackage{a4wide}
\usepackage{enumitem}
\usepackage{multicol}
\usepackage[utf8]{inputenc}

\usepackage{amsmath,amssymb,amsthm,array,textpos,float,dsfont,verbatim,fancyvrb,adjustbox,xfp,parskip}
\usepackage{breqn}

\usepackage{mathtools,makecell,graphicx,wrapfig,caption,tikz,tikz-cd,framed,pgffor}
\usetikzlibrary{arrows,positioning,decorations.pathmorphing,decorations.markings,arrows.meta, angles, calc, quotes, tikzmark,chains,shapes.geometric,bending}
\usepackage[labelformat=simple]{subcaption}

\tikzset{
  symbol/.style={
    draw=none,
    every to/.append style={
      edge node={node [sloped, allow upside down, auto=false]{$#1$}}}
  }
}
\tikzset{edge/.style= {line width=0.75pt} }
\tikzset{commutative diagrams/.cd,
	mysymbol/.style={start anchor=center,end anchor=center,draw=none}
}
\PassOptionsToPackage{hyphens}{url}\usepackage{hyperref}

\newtheoremstyle{break}
{\topsep}{\topsep}%
{\itshape}{}%
{\bfseries}{}%
{\newline}{}%

\theoremstyle{definition}
\newtheorem{definition}{Definition}[section]

\theoremstyle{plain}
\newtheorem{theorem}[definition]{Theorem}
\newtheorem{proposition}[definition]{Proposition}

\newtheorem{lemma}[definition]{Lemma}

\theoremstyle{remark}
\newtheorem{remark}[definition]{Remark}

\DeclareMathAlphabet\wiskunde{U}{msb}{m}{n}
\newcommand{\N}{{\wiskunde N}}
\newcommand{\Z}{{\wiskunde Z}}

\DeclareMathAlphabet\gotisch{U}{euf}{m}{n}

\newcommand{\Aut}{{\rm Aut}}
\newcommand{\End}{{\rm End}}

\newcommand{\Gl}{{\rm GL}}
\newcommand{\Spec}{{\rm Spec}_R}

\newcommand{\Id}{{\rm Id}}

\newcommand{\diag}{{\rm diag}}
\newcommand{\I}{\mathcal{I}}
\newcommand*{\olvarphi}{\overline\varphi}

\makeatletter
\newcommand{\biggg}{\bBigg@\thr@@}
\newcommand{\Biggg}{\bBigg@{3}}

\makeatother

\title{$R_{\infty}$--property for finitely generated torsion-free 2-step nilpotent groups of small Hirsch length}
\author{Karel Dekimpe\thanks{Supported by Methusalem grant METH/21/03 - long term structural funding of the Flemish Government.}, Maarten Lathouwers\thanks{Researcher funded by FWO PhD-fellowship fundamental research (file number: 1102424N).}\\
	Karel.Dekimpe\@@kuleuven.be, Maarten.Lathouwers\@@kuleuven.be}
\date{\today}

\begin{document}
\maketitle
\begin{abstract}
In this paper we will show that finitely generated torsion-free 2-step nilpotent groups of Hirsch length at most 6 do not have the $R_{\infty}$--property, while there are examples of such groups of Hirsch length 7 that do have the $R_{\infty}$--property.      
\end{abstract}
\section{Preliminaries}
In this paper we study the twisted conjugacy relation on 2-step nilpotent groups of small Hirsch length. To be precise, we will show that all of the finitely generated, torsion-free  2-step nilpotent groups of Hirsch length at most 6 admit an automorphism with finite Reidemeister number, while there are examples of Hirsch length 7 that have the $R_{\infty}$--property. In this first section, we will start by recalling some basic notions about finitely generated torsion-free nilpotent groups and secondly about twisted conjugacy and Reidemeister numbers, especially for these groups. 
\subsection{Hirsch length and \texorpdfstring{$\I(n,m)$}{I(n,m)} groups}
For any group $G$ we define inductively $\gamma_1(G):=G$ and $\gamma_{i+1}(G):=[\gamma_i(G),G]$ (for $i\in \N_{>0}$) to be the \textit{lower central series} of $G$. A group is \textit{c-step nilpotent} if $\gamma_{c}(G)\neq 1$ and $\gamma_{c+1}(G)=1$. Any finitely generated nilpotent group $G$ admits a series $1=G_1\lhd G_2\lhd \dots\lhd G_s=G$ with cyclic factors, i.e. $G_{i+1}/G_i$ is cyclic (see for example \cite[Theorem 17.2.2]{km79}). A group having such a subnormal series with cyclic factors is called a \textit{polycyclic} group. The \textit{Hirsch length} $h(G)$ of a polycyclic group $G$ is defined as the number of infinite cyclic factors in such a series. One can show the following properties regarding the Hirsch length.

\begin{lemma}[{\cite[page 16]{sega83}}]\label{lemma:HL properties}
	If $G$ is a polycyclic group (e.g.\ a finitely generated nilpotent group), then the Hirsch length is well-defined (i.e.\ independent of the chosen series). Moreover, if we fix some $H\subseteq G$ and $N\lhd G$, then the following holds:
        \begin{enumerate}
		\item[(i)] $h(H)\leq h(G)$
            \item[(ii)] $h(H)=h(G)\: \Longleftrightarrow \: H\subseteq_{\text{fin}} G$ (i.e. $H$ is a finite index subgroup of $G$)
		\item[(iii)] $h(G)=h(N)+h(G/N)$
		\item[(iv)] $h(G)=0\: \Longleftrightarrow\: G$ is finite
	\end{enumerate}
\end{lemma}

In this paper, we will consider 2-step nilpotent groups. For these groups, it is well known that the commutator map is bilinear.

\begin{lemma}[\cite{mks66}]\label{lemma:bilinearity [.,.] 2-nilpotent}
	If $G$ is a $2$-step nilpotent group, then $[\:.\:,.\:]:G\times G\to G$ is bilinear, i.e. for all $g_1,g_2,g_1',g_2'\in G$ it holds that 
	\[[g_1g_2,g_1'g_2']=[g_1,g_2']\cdot[g_2,g_2']\cdot[g_1,g_1']\cdot[g_2,g_1'].\]
\end{lemma}

The \textit{adapted lower central series} $\sqrt{\gamma_i(G)}$ of $G$ is defined by taking the \textit{isolator} of the lower central series, i.e.\ $\sqrt{\gamma_i(G)}:=\{g\in G\:\vert\: \exists k\in \N_{>0}:\: g^k\in \gamma_i(G)\}$. This series is central (i.e. $[\sqrt{\gamma_i(G)},G]\subset \sqrt{\gamma_{i+1}(G)}$). It will terminate in the trivial group if and only if $G$ is a torsion-free nilpotent group. Moreover, the terms are characteristic subgroups of $G$. Hence, any automorphism $\varphi\in \Aut(G)$ induces automorphisms $\varphi_i\in \Aut(\sqrt{\gamma_i(G)}/\sqrt{\gamma_{i+1}(G)})$. For finitely generated torsion-free nilpotent groups, the factors of the adapted lower central series are free abelian. Thus, in this case, the induced automorphisms $\varphi_i$ correspond to invertible integer matrices. We will address the eigenvalues of these matrices by talking about \textit{the eigenvalues of $\varphi_i$} (or more generally \textit{the eigenvalues of $\varphi$}). We will use the rank of the factors of the adapted lower central series to further divide the class of finitely generated torsion-free 2-step nilpotent groups.

\begin{definition}\label{def:I(n,m) groups}
	Let $G$ be a finitely generated torsion-free 2-step nilpotent group. We say that $G\in \I(n,m)$ (with $n,m\in \N_{>0}$) if
	\[ \frac{G}{\sqrt{\gamma_2(G)}}\cong \Z^n\quad \text{and}\quad \sqrt{\gamma_2(G)}\cong \Z^m.\]
\end{definition}
Fix some $G\in \I(n,m)$. Note that $h(G)=h(G/\sqrt{\gamma_2(G)})+h(\sqrt{\gamma_2(G)})=n+m$. Fix a $\Z$-basis $\{g_1\sqrt{\gamma_2(G)},\dots, g_n\sqrt{\gamma_2(G)}\}$ of $G/\sqrt{\gamma_2(G)}$. Since $G$ is finitely generated nilpotent, it holds that $\gamma_2(G)\subseteq_{\text{fin}} \sqrt{\gamma_2(G)}$ (by for example \cite[Lemma 2.8]{baum71}) and thus also $\gamma_2(G)\cong \Z^m$. Note that $\sqrt{\gamma_2(G)}\subset Z(G)$ (since the adapted lower central series is a central series of $G$). Using Lemma \ref{lemma:bilinearity [.,.] 2-nilpotent} this implies that $\gamma_2(G)$ is generated by $\{[g_i,g_j]\:\vert\: 1\leq i<j\leq n\}$. Hence, we obtain that
\[m=h(\gamma_2(G))\leq \# \{[g_i,g_j]\:\vert\: 1\leq i<j\leq n\}\leq n(n-1)/2.\]
Thus when considering 2-step nilpotent groups of a fixed Hirsch length $k\in \N_{>0}$, we can distinguish them by using the classes $\I(n,m)$ with $n,m\in \N_{>0}$, $n+m=k$ and $m\leq n(n-1)/2$.


\subsection{Twisted conjugacy}
Let $G$ be a group and $\varphi\in \End(G)$. We say that two elements $a,b\in G$ are \textit{$\varphi$-conjugate} or \textit{twisted conjugate} if there exists some $c\in G$ such that $a=cb\varphi(c)^{-1}$. Being twisted conjugate is an equivalence relation on $G$. The number of equivalence classes is called the \textit{Reidemeister number} $R(\varphi)$ of $\varphi$ and the collection of all Reidemeister numbers of automorphisms of $G$ is called the \textit{Reidemeister spectrum} $\Spec(G)$, i.e.
\[ \Spec(G):=\{R(\varphi)\:\vert\:\varphi\in \Aut(G)\}\subset \N_{>0}\cup \{\infty\}. \]
We say that $G$ has the \textit{$R_{\infty}$--property} if $\Spec(G)=\{\infty\}$.

In order to determine whether or not an automorphism has an infinite Reidemeister number, we will make use of the next two (well-known) theorems.

\begin{theorem}[{\cite[Lemma 2.2]{dg14}} and {\cite[Corollary 4.2]{roma11}}]\label{thm:equivalent statements R-inf}
	Let $G$ be a finitely generated torsion-free $c$-step nilpotent group and $\varphi\in \Aut(G)$, then the following are equivalent:
	\begin{enumerate}
		\item[(i)] $R(\varphi)=\infty$
		\item[(ii)] There exists some $i=1,2,\dots,c$ such that $\varphi_i\in \Aut\left(\frac{\sqrt{\gamma_i(G)}}{\sqrt{\gamma_{i+1}(G)}}\right)$ has $1$ as an eigenvalue.
		\item[(iii)] There exists some $i=1,2,\dots,c$ such that $R(\varphi_i)=\infty$.
	\end{enumerate}
\end{theorem}

\begin{theorem}[{\cite[Proposition 5]{dgo21}} and {\cite[Lemma 2.7]{roma11}}]\label{thm:RN product}
	Let $G$ be a finitely generated nilpotent group. Let
	\[G=G_1\supseteq G_2\supseteq \dots \supseteq G_{s}\supseteq G_{s+1}=1\]
	be a central series of $G$ and $\varphi\in \Aut(G)$ such that the following holds:
	\begin{enumerate}
		\item All the factors $G_i/G_{i+1}$ (with $i=1,2,\dots,s$) are torsion-free.
		\item For all terms $G_i$ (with $i=1,2,\dots,s+1$) it holds that $\varphi(G_i)=G_i$.
	\end{enumerate}
	Then it holds that:
	\[R(\varphi)=\prod_{i=1}^s R(\varphi_i)\]
	where $\varphi_i:G_i/G_{i+1}\to G_i/G_{i+1}$ (with $i=1,2,\dots,s$) are the induced automorphisms on the factor groups $G_i/G_{i+1}$.
\end{theorem}

In order to express the Reidemeister number, we will frequently use the following notation
\[|\:\cdot\:|_{\infty}:\Z\to \N_{>0}\cup \{\infty\}:x\mapsto |x|_{\infty}:=\begin{cases}
	|x| &\text{ if } x\neq 0\\
	\infty &\text{ if } x=0 \end{cases}\]
where $|\:\cdot\:|$ denotes the absolute value.
Note that for a finitely generated torsion-free nilpotent group, we can apply Theorem \ref{thm:RN product} to the adapted lower central series of $G$. Combined with the well-known description of the Reidemeister spectrum for free abelian groups of finite rank (see e.g. \cite{gw09}) we obtain the next result.

\begin{lemma} \label{lemma:RN product abelian}
	Let $G$ be a finitely generated torsion-free 2-step nilpotent group and $\varphi\in \Aut(G)$, then
	\[R(\varphi)=R(\varphi_1)R(\varphi_2)=|\det(Id-\varphi_1)|_{\infty}\:|\det(Id-\varphi_2)|_{\infty}\]
	where $\varphi_1\in \Aut(G/\sqrt{\gamma_2(G)})$ and $\varphi_2\in \Aut(\sqrt{\gamma_2(G)})$ are the induced automorphisms on the factors of the adapted lower central series.
\end{lemma}

The next lemma allows us to more easily check that an endomorphism on a finitely generated torsion-free nilpotent group is in fact an automorphism.

\begin{proposition}\label{prop:phi1 aut ==> phi aut}
	Let $G$ be a finitely generated torsion-free nilpotent group. If $\varphi\in \End(G)$ and $\varphi_1\in \Aut(G/\sqrt{\gamma_2(G)})$, then $\varphi\in \Aut(G)$.
\end{proposition}
\begin{proof}
	By (the proof of) Proposition 6.1.1 in \cite{send23}, we obtain the following result:
	\begin{equation}\label{eq:claim}
		\text{If $\psi\in \End(\Z^n)$ and $A\subset_{\text{fin}}\Z^n$ such that $\psi(A)=A$, then $\psi\in \Aut(\Z^n)$.}
	\end{equation}
	We use induction on the nilpotency class $c\in \N_{>0}$ of $G$ together with result $(\ref{eq:claim})$ to prove the proposition. If $c=1$, then there is nothing to prove, since in that case $\varphi=\varphi_1\in \Aut(G)$.\\
	Let $c>1$ and assume that the claim holds for finitely generated torsion-free nilpotent groups of nilpotency class $\leq c-1$. Consider the group $H:=G/\sqrt{\gamma_c(G)}$. Note that
	\[ \gamma_i(H)=\frac{\gamma_i(G)\sqrt{\gamma_c(G)}}{\sqrt{\gamma_c(G)}}\]
	for all $i=1,\dots,c$. Hence, $H$ is a finitely generated torsion-free nilpotent group of nilpotency class $\leq c-1$. Since $\varphi\in \End(G)$ and $\varphi(\sqrt{\gamma_c(G)})\subset \sqrt{\gamma_c(G)}$ (see for example \cite[Lemma 1.1.2 (1)]{deki96}) we can consider the induced morphism $\overline{\varphi}\in\End(H)$. By \cite[Lemma 1.1.4]{deki96} it holds that $\sqrt[H]{\gamma_2(H)}= \sqrt[G]{\gamma_2(G)}/\sqrt[G]{\gamma_c(G)}$. Hence, we obtain that
	\[ \frac{H}{\sqrt[H]{\gamma_2(H)}}=\frac{H}{\sqrt[G]{\gamma_2(G)}/\sqrt[G]{\gamma_c(G)}}\cong \frac{G}{\sqrt[G]{\gamma_2(G)}}.\]
	We consider the induced morphism $\overline{\varphi}_1\in \End(H/\sqrt{\gamma_2(H)})$. By the commutativity of the next diagram and since $\varphi_1$ is an automorphism, also $\overline{\varphi}_1$ is an automorphism.
	\[ \begin{tikzcd}
		\frac{H}{\sqrt{\gamma_2(H)}} \arrow[r, "\overline{\varphi}_1"]          & \frac{H}{\sqrt{\gamma_2(H)}}                     \\
		\frac{G}{\sqrt{\gamma_2(G)}} \arrow[r, "\varphi_1"] \arrow[u, "\cong"'] & \frac{G}{\sqrt{\gamma_2(G)}} \arrow[u, "\cong"']
		\arrow[from=1-2, to=2-1, pos=.4, phantom, "\circlearrowleft"]
	\end{tikzcd} \]
	Hence, by the induction hypothesis it follows that $\overline{\varphi}\in \Aut(G/\sqrt{\gamma_c(G)})$.\\
	Now we argue that $\varphi_c\in \Aut(\sqrt{\gamma_c(G)})$. Note that $\gamma_c(G)\lhd_{\text{fin}} \sqrt{\gamma_c(G)}\cong \Z^n$ (for some $n\in \N_{>0}$) and $\varphi(\gamma_c(G))\subset \gamma_c(G)$. Thus by result $(\ref{eq:claim})$ it suffices to prove that $\varphi(\gamma_c(G))=\gamma_c(G)$. Fix any $g_1\in G$ and $g_2\in \gamma_{c-1}(G)$. Since $\overline{\varphi}$ is an automorphism, there exists some $g_1'\in G$ and $h_1\in \sqrt{\gamma_c(G)}$ such that $g_1=\varphi(g_1')h_1$. Moreover, since $\overline{\varphi}(\gamma_{c-1}(H))=\gamma_{c-1}(H)$ and $g_2\sqrt{\gamma_c(G)}\in \gamma_{c-1}(H)$ there are $g_2'\in \gamma_{c-1}(G)\sqrt{\gamma_c(G)}$ and $h_2\in \sqrt{\gamma_c(G)}$ such that $g_2=\varphi(g_2')h_2$. Since $G$ is a finitely generated torsion-free $c$-step nilpotent group, it follows that $\sqrt{\gamma_c(G)}\subset Z(G)$. Hence, we obtain that
	\[ [g_1,g_2]=[\varphi(g_1')h_1,\varphi(g_2')h_2]=[\varphi(g_1'),\varphi(g_2')]=\varphi([g_1',g_2'])\in \varphi([G,\gamma_{c-1}(G)\sqrt{\gamma_c(G)}])=\varphi(\gamma_c(G)).\]
	So we can conclude that $\varphi(\gamma_c(G))=\gamma_c(G)$ and thus by result $(\ref{eq:claim})$ it follows that $\varphi_c\in \Aut(\sqrt{\gamma_c(G)})$.\\
    Any finitely generated torsion-free nilpotent group is \textit{Hopfian} (i.e.\ any epimorphism is an automorphism) and thus it suffices to prove that $\varphi$ is surjective. Fix any $g\in G$. Since $\overline{\varphi}$ is an automorphism, there exists some $g'\in G$ such that $\varphi(g')\sqrt{\gamma_c(G)}=g\sqrt{\gamma_c(G)}$. Using that $\varphi_c\in \Aut(\sqrt{\gamma_c(G)})$ we can take some $h\in \sqrt{\gamma_c(G)}$ such that $g=\varphi(g')\varphi(h)=\varphi(g'h)$. Hence, $\varphi$ is a surjective morphism and thus it is an automorphism.
\end{proof}

\section{Reidemeister spectrum of groups in \texorpdfstring{$\I(n,1)$}{I(n,1)}}
In this section, we determine the Reidemeister spectrum of groups in $\I(n,1)$. 

\begin{proposition}[{\cite[Proposition 5 (p. 265)]{sega83}}]\label{prop:description I(n,1)}
	Let $n\in \N_{>1}$ and $G\in\I(n,1)$, then $G$ is isomorphic to exactly one of the following groups:
	\begin{enumerate}
		\item[(i)] If $n=2r$:
		\[ G(d_1,d_2,\dots,d_r):=\Bigg\langle x_1,\dots,x_r,y_1,\dots,y_r,z\: \Biggg\vert \: \begin{array}{ll}
			[x_i,y_j]=[x_i,x_j]=[y_i,y_j]=1 &\text{ if } i\neq j \\
			\relax [x_i,y_i]=z^{d_i} & \text{ for all } i \\
			z \text{ is central} & 
		\end{array} \Bigg\rangle\]
		where $d_1\vert d_2\vert\dots\vert d_r$, each $d_i\in \N$ and $d_1\neq 0$.
		\item[(ii)] If $n=2r+1$:
		\[ G(d_1,d_2,\dots,d_r)\times \Z \]
		where $d_1\vert d_2\vert\dots\vert d_r$, each $d_i\in \N$ and $d_1\neq 0$.
	\end{enumerate}
\end{proposition}

Now we determine the Reidemeister spectrum of both classes of groups, starting with the case of even $n$.

\begin{theorem}\label{thm:Spec I(2r,1)}
	Let $n\in \N_{>1}$ with $n=2r$ and $G\in\I(n,1)$, then
	\[ \Spec(G)=2\N_{>0}\cup \{\infty\}. \]
\end{theorem}
\begin{proof}
	By Proposition \ref{prop:description I(n,1)} $(i)$ we can assume that $G=G(d_1,\dots,d_r)$ with $d_1\vert d_2\vert\dots\vert d_r$, each $d_i\in \N$ and $d_1\neq 0$. Note that $\sqrt{\gamma_2(G)}=\langle z\rangle \cong \Z$.\\
	Fix any automorphism $\varphi\in \Aut(G)$ wit $R(\varphi)<\infty$. By Lemma \ref{lemma:RN product abelian} we know that $R(\varphi)=R(\varphi_1)R(\varphi_2)$. Since the only finite Reidemeister number of $\Z$ is 2 (see for example \cite{roma11}), it holds that $R(\varphi_2)=2$. Hence, $R(\varphi)\in 2\N_{>0}\cup \{\infty\}$.\\
	Note that the identity automorphism has infinite Reidemeister number. Fix $k_1,\dots,k_r\in \N_{>0}$ and define the map $\varphi$ on $G$ that is induced by
	\[ \varphi(x_i)=y_i,\quad \varphi(y_i)=x_iy_i^{-k_i},\quad \varphi(z)=z^{-1}.\]
	Using Lemma \ref{lemma:bilinearity [.,.] 2-nilpotent}, we obtain that
	\[ [\varphi(x_i),\varphi(y_i)]=[y_i,x_iy_i^{-k_i}]=[y_i,x_i][y_i,y_i^{-k_i}]=z^{-d_i}=\varphi(z)^{d_i}.\]
	One can easily check that the other relations are also preserved and thus $\varphi\in \End(G)$. The matrix of $\varphi_1\in \Aut(G/\sqrt{\gamma_2(G)})$ with respect to the $\Z$-basis $\{x_1\sqrt{\gamma_2(G)},\dots,x_r\sqrt{\gamma_2(G)},y_1\sqrt{\gamma_2(G)},\dots,y_r\sqrt{\gamma_2(G)}\}$ equals
	\[ \begin{pmatrix}
		0 & \mathds{1}_r \\
		\mathds{1}_r & -\diag(k_1,\dots,k_r)
	\end{pmatrix} \]
	which is an invertible integer matrix. Hence, Proposition \ref{prop:phi1 aut ==> phi aut} tells us that $\varphi\in \Aut(G)$. Note that $\varphi_2\in \Aut(\sqrt{\gamma_2(G)})$ is the inversion map on $\Z$ and thus by Lemma \ref{lemma:RN product abelian} we obtain that
	\[ R(\varphi)=R(\varphi_1)R(\varphi_2)=\left|\det\left(\mathds{1}_n-\begin{pmatrix}
		0 & \mathds{1}_r \\
		\mathds{1}_r & -\diag(k_1,\dots,k_r)
	\end{pmatrix}\right)\right|_{\infty}|\det(1-(-1))|_{\infty}=2k_1\dots k_r.\]
	By taking for example $k_2=\dots=k_r=1$, we obtain all of $2\N_{>0}$ as Reidemeister number of an automorphism of $G$. Hence, we proved that $\Spec(G)=2\N_{>0}\cup \{\infty\}$.
\end{proof}

Now we consider the case where $n=2r+1$ is odd.

\begin{theorem}\label{thm:Spec I(2r+1,1)}
	Let $n\in \N_{>1}$ with $n=2r+1$ and $G\in\I(n,1)$. Fix $d_1,\dots,d_r\in \N$ (with $d_1\neq 0$ and $d_1|\dots|d_r$) such that $G\cong G(d_1,\dots,d_r)\times \Z$, then
	\[ \Spec(G)=\begin{cases}
		2\N_{>0}\cup \{\infty\} &\text{ if } d_r=0\\
		4\N_{>0}\cup \{\infty\} &\text{ if } d_r\neq 0
	\end{cases}. \]
\end{theorem}
\begin{proof}
	Without loss of generality we assume that $G=G(d_1,\dots,d_r)\times \Z$ and we denote with $u$ a generator of the $\Z$-factor.\\
	We first consider the case where $d_r=0$ (and thus, since $d_1\neq 0$, $r>1$). Note that then
	\begin{align*}
		G&=\Bigg\langle x_1,\dots,x_r,y_1,\dots,y_r,z,u\: \Biggg\vert \: \begin{array}{ll}
			[x_i,y_j]=[x_i,x_j]=[y_i,y_j]=1 &\text{ if } i\neq j \\
			\relax [x_i,y_i]=z^{d_i} & \text{ for all } i\leq r-1 \\
			x_r,y_r,u,z \text{ is central} & 
		\end{array} \Bigg\rangle\\
	&=G(d_1,\dots,d_{r-1})\times \Z^3
	\end{align*}
	where $\Z^3$ is generated by $x_r,y_r$ and $u$. Moreover, $\sqrt{\gamma_2(G)}=\langle z\rangle$. Hence, completely similar as in the proof of Theorem \ref{thm:Spec I(2r,1)}, we obtain that $\Spec(G)\subset 2\N_{>0}\cup \{\infty\}$. For the other inclusion, take any $k\in \N_{>0}$ arbitrary. By Theorem \ref{thm:Spec I(2r,1)} we can fix some $\psi\in \Aut(G(d_1,\dots,d_{r-1}))$ such that $R(\psi)=2$. Since $\Spec(\Z^3)=\N_{>0}\cup \{\infty\}$ (see for example \cite[Section 3]{roma11}), we can fix some $\psi'\in \Aut(\Z^3)$ such that $R(\psi')=k$. By Proposition 2.4 in \cite{send21} it follows that the automorphism $\varphi:=\psi\times \psi'\in \Aut(G)$ has Reidemeister number $R(\varphi)=R(\psi)R(\psi')=2k$. Hence, we argued that $\Spec(G)=2\N_{>0}\cup \{\infty\}$.

    We now consider the case $d_r\neq 0$ (and thus $d_i\neq 0$ for all $i$, since $d_1|\dots|d_r$). Fix some automorphism $\varphi\in \Aut(G)$ with $R(\varphi)<\infty$. Since $\sqrt{\gamma_2(G)}=\langle z\rangle$ and $R(\varphi_2)<\infty$, it must hold that $\varphi(z)=z^{-1}$ and thus $R(\varphi_2)=2$. Note that since $d_i\neq 0$ for all $i$, it holds that $Z(G)=\langle u,z\rangle\cong \Z^2$. Since $Z(G)$ is a characteristic subgroup and $\varphi\vert_{Z(G)}$ is an automorphism, it thus follows that $\varphi(u)=u^{\alpha}z^{\beta}$ with $\alpha\in\{-1,1\}$ and $\beta\in \Z$. If we denote with $\olvarphi\in \Aut(G/Z(G))$ the induced automorphism, then we obtain by applying Theorem \ref{thm:RN product} to $G\supset Z(G)\supset 1$ that
	\[ R(\varphi)=R(\varphi\vert_{Z(G)})R(\olvarphi)=R(\olvarphi)\left|\det\left(\mathds{1}_2-\begin{pmatrix}
		\alpha & 0 \\
		\beta & -1
	\end{pmatrix}\right)\right|_{\infty}=2|1-\alpha|_{\infty}R(\olvarphi).\]
	Since $R(\varphi)<\infty$, we thus obtain that $\alpha=-1$ and $R(\varphi)=4R(\olvarphi)\in 4\N_{>0}\cup \{\infty\}$.\\
	Note that the identity morphism has infinite Reidemeister number. Fix some $k_1,\dots,k_r\in \N_{>0}$ and consider the map $\varphi$ on $G$ induced by
	\[ \varphi(x_i)=y_i,\quad \varphi(y_i)=x_iy_i^{-k_i},\quad \varphi(z)=z^{-1},\quad \varphi(u)=u^{-1}.\]
	One can check that $\varphi$ preserves the relations (using Lemma \ref{lemma:bilinearity [.,.] 2-nilpotent}) and thus it induces an endomorphism of $G$. The matrix of the induced map $\varphi_1\in \Aut(G/\sqrt{\gamma_2(G)})$ with respect to the $\Z$-basis $\{x_1\sqrt{\gamma_2(G)},\dots,x_r\sqrt{\gamma_2(G)},y_1\sqrt{\gamma_2(G)},\dots,y_r\sqrt{\gamma_2(G)},u\sqrt{\gamma_2(G)}\}$ equals
	\[ \begin{pmatrix}
		0 & \mathds{1}_r & 0 \\
		\mathds{1}_r & -\diag(k_1,\dots,k_r) & 0\\
		0 & 0 & -1
	\end{pmatrix} \]
	which is an invertible integer matrix. Hence, Proposition \ref{prop:phi1 aut ==> phi aut} tells us that $\varphi\in \Aut(G)$. By applying Lemma \ref{lemma:RN product abelian} we obtain that
	\[ R(\varphi)=R(\varphi_1)R(\varphi_2)=\left|\det\left(\mathds{1}_n-\begin{pmatrix}
		0 & \mathds{1}_r & 0 \\
		\mathds{1}_r & -\diag(k_1,\dots,k_r) & 0\\
		0 & 0 & -1
	\end{pmatrix}\right)\right|_{\infty}|1-(-1)|_{\infty}=4k_1\dots k_r.\]
	We can now for example take $k_2=\dots=k_r=1$ to obtain an automorphism $\varphi\in \Aut(G)$ with $R(\varphi)=4k_1$. Since $k_1\in \N_{>0}$ can be taken arbitrary, we proved that $\Spec(G)=4\N_{>0}\cup \{\infty\}$.
\end{proof}

\section{Hirsch length 5 or less}
Recall that if $G$ is a finitely generated torsion-free 2-step nilpotent group, then $G\in I(n,m)$ for some $n,m>0$ with $h(G)=n+m$ and $m\leq n(n-1)/2$. Hence, there do not exist such groups of Hirsch length 1 or 2 and the only ones of Hirsch length 3 and 4 belong to $\I(2,1)$ and $\I(3,1)$ respectively. By Theorem \ref{thm:Spec I(2r,1)} and \ref{thm:Spec I(2r+1,1)} we know that those of Hirsch length 3 have Reidemeister spectrum $2\N_{>0}\cup \{\infty\}$ and those of Hirsch length 4 have Reidemeister spectrum $4\N_{>0}\cup \{\infty\}$. Again by Theorem \ref{thm:Spec I(2r,1)}, we already know that all groups in $\I(4,1)$ have Reidemeister spectrum $2\N_{>0}\cup\{\infty\}$. Thus to understand the situation up to Hirsch length 5, we only need to consider the groups of $\I(3,2)$. The next lemma gives a presentation for any of those groups. In this section, we shorten the notation of the presentations and omit all trivial commutators of generators. E.g.\ in the statement of the following lemma, it is implicitly understood that $[x_2,x_3]=1$ and also $[x_i,z_j]=1$ (for $i=1,2,3$ and $j=1,2$).

\begin{lemma}
\label{lem:presentation I(3,2)}
	Let $G\in \I(3,2)$, then it can be presented via
	\[ G=\langle x_1,x_2,x_3,z_1,z_2\:\vert\: [x_1,x_2]=z_1^{\alpha}z_2^{\beta},[x_1,x_3]=z_2^{\gamma}\rangle\]
	with $\alpha,\gamma\in \Z\setminus\{0\}$ and $\beta\in \Z$.
\end{lemma}
\begin{proof}
	Since $G\in \I(3,2)$ it has a presentation
	\[ G=\langle x_1,x_2,x_3,z_1,z_2\:\vert\: [x_1,x_2]=z_1^{s_{12}}z_2^{t_{12}},[x_1,x_3]=z_1^{s_{13}}z_2^{t_{13}},[x_2,x_3]=z_1^{s_{23}}z_2^{t_{23}}\rangle\]
	where $s_{ij},t_{ij}\in \Z$. Hence, $G/\langle z_2\rangle\in \I(3,1)$ and by 
 Proposition~\ref{prop:description I(n,1)} we may assume that  $G/\langle z_2\rangle$ has the following presentation
	\[ G/\langle z_2\rangle=\langle \overline{x_1},\overline{x_2},\overline{x_3},\overline{z_1}\:\vert\: [\overline{x_1},\overline{x_2}]=\overline{z_1}^{\alpha}\rangle \]
 for some $\alpha\neq 0$.
	Thus, we obtain that $G$ has a presentation (with $\beta = t_{12}$) of the form
	\[ G=\langle x_1,x_2,x_3,z_1,z_2\:\vert\: [x_1,x_2]=z_1^{\alpha}z_2^{\beta},[x_1,x_3]=z_2^{t_{13}},[x_2,x_3]=z_2^{t_{23}}\rangle.\]
	Now we can assume that $t_{13}$ and $t_{23}$ are both non-zero. Indeed, if only one of both is zero, the result trivially holds (by swapping the roles of $x_1$ and $x_2$ if necessary). If $t_{13}=t_{23}=0$, then one can prove that $G\not\in \I(3,2)$ (by for example using Proposition 6.2.3 in \cite{deki96}). Moreover, we can assume that $t_{13}>0$ by inverting the generator $z_2$ if necessary. For two integers $a,b\in \Z$ with $b> 0$, we denote with $a\text{ mod } b$ the remainder of the division of $a$ by $b$, i.e. $0\leq a\text{ mod } b < b$ and there exists a unique $q\in \Z$ such that $a=qb+(a\text{ mod } b)$. In particular, there exists a unique $k\in \Z$ such that $t_{23}+kt_{13}=(t_{23}\text{ mod } t_{13})$. We fix a new set of generators of $G$
	\[ x_1':=x_1,\: x_2':=x_1^kx_2,\: x_3':=x_3,\: z_1':=z_1,\: z_2':=z_2.\]
	Using the bilinearity of the commutator (see Lemma \ref{lemma:bilinearity [.,.] 2-nilpotent}) one can easily check that the new presentation (where we omit the accents) of $G$ becomes
	\[ G=\langle x_1,x_2,x_3,z_1,z_2\:\vert\: [x_1,x_2]=z_1^{\alpha}z_2^{\beta},[x_1,x_3]=z_2^{t_{13}},[x_2,x_3]=z_2^{t_{23}\text{ mod } t_{13}}\rangle.\]
	If $t_{23}\text{ mod } t_{13}=0$ then we are done. Otherwise, we take $l\in \Z$ such that $t_{13}+l(t_{23}\text{ mod } t_{13})=(t_{13}\text{ mod } (t_{23}\text{ mod } t_{13}))$ and define the new set of generators of $G$ by
	\[ x_1':=x_2^lx_1,\: x_2':=x_2,\: x_3':=x_3,\: z_1':=z_1,\: z_2':=z_2.\]
	It is easy to verify that now $G$ has the following presentation (where we omit the accents)
	\[ G=\langle x_1,x_2,x_3,z_1,z_2\:\vert\: [x_1,x_2]=z_1^{\alpha}z_2^{\beta},[x_1,x_3]=z_2^{t_{13}\text{ mod } (t_{23}\text{ mod } t_{13})},[x_2,x_3]=z_2^{t_{23}\text{ mod } t_{13}}\rangle.\]
	We can keep on reducing the powers of $z_2$ in the presentation using these two tricks. Eventually, this procedure will end (since it is nothing else than Euclid's algorithm on the two occurring powers of $z_2$). Hence, $G$ will have one of the following presentations
	\begin{align*}
		G&=\langle x_1,x_2,x_3,z_1,z_2\:\vert\: [x_1,x_2]=z_1^{\alpha}z_2^{\beta},[x_1,x_3]=z_2^{\gcd(t_{13},t_{23})}\rangle \text{ or }\\
		G&=\langle x_1,x_2,x_3,z_1,z_2\:\vert\: [x_1,x_2]=z_1^{\alpha}z_2^{\beta},[x_2,x_3]=z_2^{\gcd(t_{13},t_{23})}\rangle
	\end{align*}
	Define $\gamma:=\gcd(t_{13},t_{23})$ and remark that $\gamma\neq 0$ (since $t_{13}$ and $t_{23}$ are both non-zero). Thus, in the first case we are done. In the second case, by taking the following new generators for $G$
	\[ x_1':=x_2,\: x_2':=x_1^{-1},\: x_3':=x_3,\: z_1':=z_1,\: z_2':=z_2\]
	we obtain the desired presentation.
\end{proof}

The description in Lemma \ref{lem:presentation I(3,2)} allows us to describe for each group in $\I(3,2)$ an automorphism with finite Reidemeister number.

\begin{theorem}\label{thm:Spec I(3,2)}
	Let $G\in \I(3,2)$, then $G$ does not have the $R_{\infty}$--property.
\end{theorem}
\begin{proof}
	By Lemma \ref{lem:presentation I(3,2)} we know that
	\[ G=\langle x_1,x_2,x_3,z_1,z_2\:\vert\: [x_1,x_2]=z_1^{\alpha}z_2^{\beta},[x_1,x_3]=z_2^{\gamma}\rangle\]
	with $\alpha,\gamma\in \Z\setminus\{0\}$ and $\beta\in \Z$. Take $k,l\in \Z\setminus\{0\}$ such that $4+\alpha\gamma^2 kl\neq 0$ and define the map $\varphi$ on $G$ that is induced by
	\[ \varphi(x_1):=x_1^{-1},\quad \varphi(x_2):=x_2^{-\beta\gamma l-\alpha\gamma^2 kl-1}x_3^{\beta^2 l+\alpha\beta\gamma kl-\alpha k},\quad \varphi(x_3):=x_2^{-\gamma^2 l}x_3^{\beta\gamma l-1}, \]
	\[ \varphi(z_1):=z_1^{\alpha\gamma^2 kl+1}z_2^{\gamma k},\quad \varphi(z_2):=z_1^{\alpha\gamma l}z_2.\]
	Using the bilinearity of the commutator (see Lemma \ref{lemma:bilinearity [.,.] 2-nilpotent}) one can check that the relations are preserved by $\varphi$. Hence, $\varphi$ induces a morphism on $G$. Moreover, the matrix of the induced morphism $\varphi_1\in \End(G/\sqrt{\gamma_2(G)})$ with respect to the $\Z$-basis $\{x_1\sqrt{\gamma_2(G)},x_2\sqrt{\gamma_2(G)},x_3\sqrt{\gamma_2(G)}\}$ equals
	\[ \begin{pmatrix}
		-1 & 0 & 0 \\
		0 & -\beta\gamma l-\alpha\gamma^2 kl-1 & -\gamma^2 l \\
		0 & \beta^2 l+\alpha\beta\gamma kl-\alpha k & \beta\gamma l-1
	\end{pmatrix} \]
	and is an invertible matrix over $\Z$. Hence, Proposition \ref{prop:phi1 aut ==> phi aut} tells us that $\varphi$ is an automorphism of $G$. Using Lemma \ref{lemma:RN product abelian}, one can now calculate that
	\[ R(\varphi)=R(\varphi_1)R(\varphi_2)=(2|4+\alpha\gamma^2 kl|_{\infty})(|\alpha\gamma^2 kl|_{\infty})=2|\alpha\gamma^2 kl(4+\alpha\gamma^2 kl)|_{\infty}<\infty \]
	where we used the assumption that $4+\alpha\gamma^2 kl\neq 0$ with $k,l\in \Z\setminus\{0\}$.
\end{proof}

\section{Hirsch length 6}
By Theorem \ref{thm:Spec I(2r+1,1)} we know that all groups in $\I(5,1)$ have Reidemeister spectrum equal to $2\N_{>0}\cup \{\infty\}$ or $4\N_{>0}\cup\{\infty\}$. The only other classes to consider are $\I(3,3)$ and $\I(4,2)$. Malfait argued in \cite[Theorem 3.8]{malf00} that all groups in $\I(3,3)$ admit a \textit{hyperbolic} automorphism.

\begin{definition}
	Let $G$ be a finitely generated torsion-free c-step nilpotent group, then $\varphi\in \Aut(G)$ is called a hyperbolic automorphism if it has no eigenvalues of absolute value one.
\end{definition}

In particular, Theorem \ref{thm:equivalent statements R-inf} tells us that a hyperbolic automorphism has a finite Reidemeister number.

\begin{theorem}[{\cite[Theorem 3.8]{malf00}}]\label{thm:Spec I(3,3)}
	Any group $G\in \I(3,3)$ admits a hyperbolic automorphism. In particular, no group in $\I(3,3)$ has the $R_{\infty}$--property.
\end{theorem}

\begin{remark}
	In the article by Malfait, they also prove (see \cite[Proposition 3.1]{malf00}) that each automorphism of a group in $\I(3,2)$ has an eigenvalue of absolute value one (and thus the group does not admit a hyperbolic automorphism). Nevertheless, in Theorem \ref{thm:Spec I(3,2)} we constructed for each group in $\I(3,2)$ an automorphism that does not have one as an eigenvalue (or equivalently has finite Reidemeister number). 
\end{remark}

What rests is to consider the groups in $\I(4,2)$. We start by describing the classification of all the groups in $\I(4,2)$. More details about this classification can be found in \cite{gss82}. Let $G\in \I(4,2)$. Note that $\sqrt{\gamma_2(G)}/\gamma_2(G)$ is a finitely generated abelian torsion group. Hence, it follows that
\[ \frac{\sqrt{\gamma_2(G)}}{\gamma_2(G)}\cong \frac{\Z}{\delta\Z}\times \frac{\Z}{\lambda\delta\Z} \]
where $\delta,\lambda\in \N_{>0}$. We define $\delta(G):=\delta$ and $\lambda(G):=\lambda$ to be these unique integers. Besides the invariants $\delta(G)$ and $\lambda(G)$, we also need the notion of \textit{$\lambda$-equivalent} binary quadratic forms for the classification of the groups in $\I(4,2)$.  

\begin{definition}\label{def:equivalent binary quadratic forms}
	Let $\Phi$ and $\Psi$ be two binary quadratic forms over $\Z$ and $\lambda\in \Z\setminus\{0\}$ some non-zero integer. Then $\Phi$ and $\Psi$ are called $\lambda$-equivalent, denoted by $\Phi\overset{\lambda}{\sim}\Psi$, if there exists an invertible integer matrix
	\[ \begin{pmatrix}
		a & b \\
		\lambda c & d
	\end{pmatrix}\in \Gl_2(\Z), \text{ with } c\in \Z\]
such that
\[ \Psi(X,Y)=\pm \Phi(aX+bY,\lambda cX+dY). \]
\end{definition}

One can easily verify that the relation ``$\lambda$-equivalence'' is indeed an equivalence relation on the set of binary quadratic forms over $\Z$.

Let $\Phi(X,Y)=aX^2+bXY+cY^2$ be a binary quadratic form over $\Z$ and $\delta,\lambda\in \N_{>0}$. We define the group $G(\delta,\lambda,\Phi)$ by
\[ G(\delta,\lambda,\Phi):=\Bigg\langle x_1,x_2,x_3,x_4,z_1,z_2\: \Biggg\vert \: \begin{array}{l}
	[x_1,x_3]=z_2^{\delta\lambda},\: [x_1,x_4]=z_1^{\delta}\\
	\relax [x_2,x_3]=z_1^{a\delta}z_2^{b\delta\lambda},\: [x_2,x_4]=z_2^{-c\delta\lambda}\\
	\relax [x_1,x_2]=[x_3,x_4]=1 \text{ and } z_1,\: z_2 \text{ central}
\end{array} \Bigg\rangle.\]
Note that $G(\delta,\lambda,\Phi)\in \I(4,2)$, $\lambda(G(\delta,\lambda,\Phi))=\lambda$ and $\delta(G(\delta,\lambda,\Phi))=\delta$.
Grunewald, Segal and Sterling gave a classification of the groups in $\I(4,2)$ using $\lambda$-equivalent binary quadratic forms and the groups $G(\delta,\lambda,\Phi)$ defined above.

\begin{theorem}[{\cite[Theorem 1]{gss82}}]\label{thm:presentation I(4,2)}
	Let $\delta,\lambda\in \N_{>0}$. The assignment
	\[ \Phi\mapsto G(\delta,\lambda,\Phi) \]
	induces a bijective correspondence between the set of $\lambda$-equivalence classes of binary quadratic forms over $\Z$ and the set of isomorphism classes of groups $G\in \I(4,2)$ with $\lambda(G)=\lambda$ and $\delta(G)=\delta$.\\
	In other words, any group $G\in\I(4,2)$ is isomorphic to some $G(\delta(G),\lambda(G),\Phi)$ and two such groups $G(\delta,\lambda,\Phi)$ and $G(\delta',\lambda',\Psi)$ are isomorphic if and only if $\delta=\delta'$, $\lambda=\lambda'$ and $\Phi\overset{\lambda}{\sim}\Psi$.
\end{theorem}

Using the classification in Theorem \ref{thm:presentation I(4,2)}, we are now able to describe for each group in $\I(4,2)$ an automorphism with a finite Reidemeister number.

\begin{theorem}\label{thm:Spec I(4,2)}
	Let $G\in \I(4,2)$, then $G$ does not have the $R_{\infty}$--property.
\end{theorem}
\begin{proof}
	By Theorem \ref{thm:presentation I(4,2)} we can assume without loss of generality that $G=G(\delta,\lambda,\Phi)$ with $\delta,\lambda\in \N_{>0}$ and $\Phi(X,Y)=aX^2+bXY+cY^2$ a binary quadratic form over $\Z$.\\
    Let us first assume that $\Phi(0,1)=c\neq 0$. Consider the map $\varphi$ on $G$ that is induced by
	\[ \varphi(x_1):=x_1^{-1-4c}x_4^2,\quad \varphi(x_2):=x_2^{-1-4c}x_3^{-2c},\quad \varphi(x_3):=x_2^2x_3,\quad \varphi(x_4):=x_1^{-2c}x_4, \]
	\[ \varphi(z_1):=z_1^{-1},\quad \varphi(z_2):=z_2^{-1}.\]
	One can check using the bilinearity of the commutator (see Lemma \ref{lemma:bilinearity [.,.] 2-nilpotent}) that all relations are preserved by $\varphi$ and thus it induces a morphism on $G$. Moreover, the matrix of $\varphi_1\in \End(G/\sqrt{\gamma_2(G)})$ with respect to the $\Z$-basis $\{x_1\sqrt{\gamma_2(G)},x_2\sqrt{\gamma_2(G)},x_3\sqrt{\gamma_2(G)},x_4\sqrt{\gamma_2(G)}\}$ equals
	\[ \begin{pmatrix}
		-1-4c & 0 & 0 & -2c \\
		0 & -1-4c & 2 & 0 \\
		0 & -2c & 1 & 0 \\
		2 & 0 & 0 & 1
	\end{pmatrix} \]
	and is an invertible matrix over $\Z$. Proposition \ref{prop:phi1 aut ==> phi aut} now tells us that $\varphi$ is an automorphism of $G$. Using Lemma \ref{lemma:RN product abelian} we obtain that
	\[ R(\varphi)=R(\varphi_1)R(\varphi_2)=(16|c^2|_{\infty})4=64|c^2|_{\infty}<\infty \]
	where we used that $c\neq 0$.

	Assume that $c=0$. Note that if also $a=b=0$, then
	\begin{align*}
		G(\delta,\lambda,0)&=\Bigg\langle x_1,x_2,x_3,x_4,z_1,z_2\: \Biggg\vert \: \begin{array}{l}
			[x_1,x_3]=z_2^{\delta\lambda},\: [x_1,x_4]=z_1^{\delta}\\
			\relax [x_3,x_4]=1 \text{ and } x_2,\: z_1,\: z_2 \text{ central}
		\end{array} \Bigg\rangle\\
		&=\Bigg\langle x_1,x_3,x_4,z_1,z_2\: \Biggg\vert \: \begin{array}{l}
			[x_1,x_3]=z_2^{\delta\lambda},\: [x_1,x_4]=z_1^{\delta}\\
			\relax [x_3,x_4]=1 \text{ and } z_1,\: z_2 \text{ central}
		\end{array} \Bigg\rangle\times \Z\\
		&=:G'\times \Z
	\end{align*}
	where $G'\in \I(3,2)$. By Theorem \ref{thm:Spec I(3,2)} we can take some $\psi'\in \Aut(G')$ with $R(\psi')<\infty$. Consider now the automorphism $\varphi:=\psi'\times -\Id_{\Z}\in \Aut(G)$. Using Proposition 2.4 in \cite{send21}, it it follows that $R(\varphi)=R(\psi')R(-\Id_{\Z})=2R(\psi')<\infty$.\\
    So we can assume without loss of generality that $a$ and $b$ are not both zero (but $c=0$). We now argue that in this case, $G$ is isomorphic to some $G(\delta,\lambda,\Psi)$ with $\Psi(0,1)\neq 0$. For this, define $\Psi_k(X,Y):=\Phi((k\lambda +1)X+kY,\lambda X+Y)$ for any $k\in \Z$. Note that the matrix 
	\[\begin{pmatrix}
		k\lambda +1 & k \\
		\lambda & 1
	\end{pmatrix}\in \Gl_2(\Z)\]
	and thus $\Phi\overset{\lambda}{\sim}\Psi_k$ for all $k\in \Z$. Since we know that $\Phi(X,Y)=aX^2+bXY$, some calculation yields that
	\[ \Psi_k(X,Y)=\left(a(k\lambda+1)^2+b\lambda(k\lambda+1)\right)X^2+\left(2ak(k\lambda+1)+b(2k\lambda+1)\right)XY+k\left(ak+b\right)Y^2. \]
	Since not both $a$ and $b$ are zero, we can take some $k_0\in \Z$ such that $k_0(ak_0+b)\neq 0$. Hence, $G\cong G(\delta,\lambda,\Psi_{k_0})$ and $\Psi_{k_0}(0,1)=k_0(ak_0+b)\neq 0$. Thus we reduced this case to the first one and we get an automorphism with a finite Reidemeister number.
\end{proof}

\section{Conclusion and minimal example}
In the previous sections, we argued that all finitely generated torsion-free 2-step nilpotent groups of Hirsch length at most 6 do not have the $R_{\infty}$--property. In Remark 7.8 in \cite{dl23} we gave an example of a finitely generated torsion-free 2-step nilpotent group of Hirsch length 7 that has the $R_{\infty}$--property.

\begin{theorem}
	All finitely generated torsion-free 2-step nilpotent groups of Hirsch length at most 6 do not have the $R_{\infty}$--property. This upper bound is sharp, i.e. there exists a finitely generated torsion-free 2-step nilpotent group of Hirsch length 7 with the $R_{\infty}$--property.
\end{theorem}
\begin{proof}
	Combining Theorems \ref{thm:Spec I(2r,1)}, \ref{thm:Spec I(2r+1,1)}, \ref{thm:Spec I(3,2)}, \ref{thm:Spec I(3,3)} and \ref{thm:Spec I(4,2)} yields that there do not exist finitely generated torsion-free 2-step nilpotent groups of Hirsch length at most 6 with the $R_{\infty}$--property.\\
    We now give an example of such a group of Hirsch length 7 with the $R_{\infty}$--property. Consider the group $G$ defined by
	\[ G:=\Bigg\langle x_1,x_2,x_3,x_4,z_1,z_2,z_3\: \Biggg\vert \: \begin{array}{l}
		[x_1,x_2]=z_1,\: [x_2,x_3]=z_2,\: [x_3,x_4]=z_3,\\
		\relax [x_1,x_3]=[x_1,x_4]=[x_2,x_4]=1 \text{ and }  z_1,\: z_2,\: z_3 \text{ central}
	\end{array} \Bigg\rangle.\]
	Note that $G\in \I(4,3)$ is indeed a finitely generated torsion-free 2-step nilpotent group with $h(G)=7$. In \cite[Remark 7.8]{dl23} we argue that $G$ has the $R_{\infty}$--property by using that this group can be associated to the path graph on 4 vertices. In \cite[Example 4.1]{gw09}, it is proven directly that $G$ has the $R_{\infty}$--property.
\end{proof}

\medskip

\begin{remark}
For higher nilpotency classes the analogous problem is quite trivial. If $G$ is a finitely generated torsion-free nilpotent group of class $c$, then the Hirsch length of $G$ is at least $c+1$. It can be easily shown that for any $c\geq 3$ there exists a finitely generated torsion-free $c$-step nilpotent group of Hirsch length $c+1$ with the $R_{\infty}$--property.

In fact, the case $c=3$ can be proved in a similar way as Example 5.2 of \cite{gw09}, while the claim for such groups $G$ of higher nilpotency classes then reduces to the case $c=3$ by considering $G/\gamma_4(G)$.

\end{remark}

\bibliographystyle{alpha}

\begin{thebibliography}{DGO21}

\bibitem[Bau71]{baum71}
Gilbert Baumslag.
\newblock {\em Lecture notes on nilpotent groups}, volume No. 2 of {\em
  Regional Conference Series in Mathematics}.
\newblock American Mathematical Society, Providence, RI, 1971.

\bibitem[Dek96]{deki96}
Karel Dekimpe.
\newblock {\em Almost-{B}ieberbach groups: affine and polynomial structures},
  volume 1639 of {\em Lecture Notes in Mathematics}.
\newblock Springer-Verlag, Berlin, 1996.

\bibitem[DG14]{dg14}
Karel Dekimpe and Daciberg Gon\c{c}alves.
\newblock The {$R_\infty$}-property for free groups, free nilpotent groups and
  free solvable groups.
\newblock {\em Bull. Lond. Math. Soc.}, 46(4):737--746, 2014.

\bibitem[DGO21]{dgo21}
Karel Dekimpe, Daciberg~Lima Gon\c{c}alves, and Oscar Ocampo.
\newblock The {$R_\infty$} property for pure {A}rtin braid groups.
\newblock {\em Monatsh. Math.}, 195(1):15--33, 2021.

\bibitem[DL23]{dl23}
Karel Dekimpe and Maarten Lathouwers.
\newblock The {R}eidemeister spectrum of 2-step nilpotent groups determined by
  graphs.
\newblock {\em Comm. Algebra}, 51(6):2384--2407, 2023.

\bibitem[GSS82]{gss82}
Fritz~J. Grunewald, Daniel Segal, and Leon~S. Sterling.
\newblock Nilpotent groups of {H}irsch length six.
\newblock {\em Math. Z.}, 179(2):219--235, 1982.

\bibitem[GW09]{gw09}
Daciberg Gon\c{c}alves and Peter Wong.
\newblock Twisted conjugacy classes in nilpotent groups.
\newblock {\em J. Reine Angew. Math.}, 633:11--27, 2009.

\bibitem[KM79]{km79}
M.~I. Kargapolov and Ju.~I. Merzljakov.
\newblock {\em Fundamentals of the theory of groups}, volume~62 of {\em
  Graduate Texts in Mathematics}.
\newblock Springer-Verlag, New York-Berlin, russian edition, 1979.

\bibitem[Mal00]{malf00}
Wim Malfait.
\newblock Anosov diffeomorphisms on nilmanifolds of dimension at most six.
\newblock {\em Geom. Dedicata}, 79(3):291--298, 2000.

\bibitem[MKS66]{mks66}
Wilhelm Magnus, Abraham Karrass, and Donald Solitar.
\newblock {\em Combinatorial group theory: {P}resentations of groups in terms
  of generators and relations}.
\newblock Interscience Publishers [John Wiley \& Sons], New York-London-Sydney,
  1966.

\bibitem[Rom11]{roma11}
V.~Roman'kov.
\newblock Twisted conjugacy classes in nilpotent groups.
\newblock {\em J. Pure Appl. Algebra}, 215(4):664--671, 2011.

\bibitem[Seg83]{sega83}
Daniel Segal.
\newblock {\em Polycyclic groups}, volume~82 of {\em Cambridge Tracts in
  Mathematics}.
\newblock Cambridge University Press, Cambridge, 1983.

\bibitem[Sen21]{send21}
Pieter Senden.
\newblock Twisted conjugacy in direct products of groups.
\newblock {\em Comm. Algebra}, 49(12):5402--5422, 2021.

\bibitem[Sen23]{send23}
Pieter Senden.
\newblock {\em How does the structure of a group determine its {R}eidemeister
  spectrum?}
\newblock PhD thesis, KU Leuven, Leuven, 2023.

\end{thebibliography}

\end{document}